%% file: non_finite_open_modular_functors.tex
\documentclass[10pt,a4paper]{scrartcl}

\usepackage[british]{babel}

\usepackage[a4paper,top=3cm,bottom=3cm,left=3cm,right=3cm]{geometry}

\usepackage{graphicx} 
\usepackage{amsmath}
\usepackage{amssymb}
\usepackage{amsthm}
\usepackage{stmaryrd}
\usepackage{mathrsfs}

\makeatletter
\DeclareRobustCommand{\em}{%
	\@nomath\em \if b\expandafter\@car\f@series\@nil
	\normalfont \else \slshape \fi}
\makeatother

\usepackage{needspace}

\usepackage[utf8]{inputenc}
\usepackage{color}
\usepackage{enumitem}
\usepackage[affil-it]{authblk}
\usepackage{graphicx} 

\usepackage{tikz-cd}
\tikzstyle{tikzfig}=[baseline=-0.25em,scale=0.5]
\let\to\undefined\newcommand{\to}{\longrightarrow}
\let\mapsto\undefined\newcommand{\mapsto}{\longmapsto}
\pgfkeys{/tikz/tikzit fill/.initial=0}
\pgfkeys{/tikz/tikzit draw/.initial=0}
\pgfkeys{/tikz/tikzit shape/.initial=0}
\pgfkeys{/tikz/tikzit category/.initial=0}

\pgfdeclarelayer{edgelayer}
\pgfdeclarelayer{nodelayer}
\pgfsetlayers{background,edgelayer,nodelayer,main}

\tikzstyle{none}=[inner sep=0mm]

\newcommand{\tikzfig}[1]{%
	{\tikzstyle{every picture}=[tikzfig]
		\IfFileExists{#1.tikz}
		{\input{#1.tikz}}
		{%
			\IfFileExists{./figures/#1.tikz}
			{\input{./figures/#1.tikz}}
			{\tikz[baseline=-0.5em]{\node[draw=red,font=\color{red},fill=red!10!white] {\textit{#1}};}}%
	}}%
}

\tikzstyle{every loop}=[]

\usepackage{tikzit}

\input{styles.tikzstyles}
\usetikzlibrary{decorations.pathreplacing,decorations.markings}
\tikzset{
	on each segment/.style={
		decorate,
		decoration={
			show path construction,
			moveto code={},
			lineto code={
				\path [#1]
				(\tikzinputsegmentfirst) -- (\tikzinputsegmentlast);
			},
			curveto code={
				\path [#1] (\tikzinputsegmentfirst)
				.. controls
				(\tikzinputsegmentsupporta) and (\tikzinputsegmentsupportb)
				..
				(\tikzinputsegmentlast);
			},
			closepath code={
				\path [#1]
				(\tikzinputsegmentfirst) -- (\tikzinputsegmentlast);
			},
		},
	},
	mid arrow/.style={postaction={decorate,decoration={
				markings,
				mark=at position .7 with {\arrow[#1]{stealth}}
	}}},
}
\tikzset{%
	link/.style    = { white, double = black, line width = 1.8pt,
		double distance = 0.4pt },
	channel/.style = { white, double = black, line width = 0.8pt,
		double distance = 0.8pt },
}

\tikzset{%
	blink/.style    = { white, double = blue, line width = 2pt,
		double distance = 1pt },
	channel/.style = { white, double = blue, line width = 2pt,
		double distance = 1pt },
}

\tikzstyle{tikzfig}=[baseline=-0.25em,scale=0.5]

\pgfkeys{/tikz/tikzit fill/.initial=0}
\pgfkeys{/tikz/tikzit draw/.initial=0}
\pgfkeys{/tikz/tikzit shape/.initial=0}
\pgfkeys{/tikz/tikzit category/.initial=0}

\pgfdeclarelayer{edgelayer}
\pgfdeclarelayer{nodelayer}
\pgfsetlayers{background,edgelayer,nodelayer,main}

\tikzstyle{none}=[inner sep=0mm]

\tikzstyle{every loop}=[]

\newtheoremstyle{dtheorem}
{\topsep}
{\topsep}
{\slshape}
{0pt}
{\bfseries}
{.}
{ }
{\thmname{#1}\thmnumber{ #2}\thmnote{ {\normalfont\slshape(#3)}}}

\newtheoremstyle{ddefinition}
{\topsep}
{\topsep}
{\normalfont}
{0pt}
{\bfseries}
{.}
{ }
{\thmname{#1}\thmnumber{ #2}\thmnote{ {\normalfont\slshape(#3)}}}

\theoremstyle{dtheorem}
\newtheorem{theorem}{Theorem}[section]

\makeatletter
\newtheorem*{rep@theorem}{\rep@title}
\newcommand{\newreptheorem}[2]{%
	\newenvironment{rep#1}[1]{%
		\def\rep@title{#2 \ref{##1}}%
		\begin{rep@theorem}}%
		{\end{rep@theorem}}}
\makeatother

\newreptheorem{theorem}{Theorem}
\newreptheorem{corollary}{Corollary}
\newtheorem{lemma}[theorem]{Lemma}
\newtheorem{proposition}[theorem]{Proposition}
\newtheorem{corollary}[theorem]{Corollary}

\theoremstyle{ddefinition}

\newtheorem{exxample}[theorem]{Example}

\numberwithin{equation}{section}

\newcommand{\cat}{\mathcal{A}}
\newcommand{\Hom}{\mathsf{Hom}_{\mathcal{A}}}

\usepackage[hidelinks]{hyperref}
\usepackage{mathtools}
\mathtoolsset{showonlyrefs}

\setkomafont{section}{\rmfamily\large}
\setkomafont{subsection}{\normalfont\slshape}
\setkomafont{subsubsection}{\rmfamily}
\setkomafont{subparagraph}{\normalfont\scshape}

\newtheorem*{theorem*}{Theorem}
\newtheorem*{corollary*}{Corollary}

\usepackage{titlesec}
\titleformat{\subsection}[runin]
{\normalfont\slshape}
{\thesubsection}
{0.5em}
{}
[.]

\usepackage{titletoc}
\dottedcontents{section}[1.5em]{\rmfamily}{1.5em}{0.2cm}
\dottedcontents{subsection}[5em]{\rmfamily}{2.3em}{0.2cm}

\makeatletter
\renewcommand\section{\@startsection {section}{1}{\z@}%
	{-3.5ex \@plus -1ex \@minus -.2ex}%
	{2.3ex \@plus.2ex}%
	{\normalfont\scshape\centering}}
\makeatother

\begin{document}
	\vspace*{-1.5cm}	\begin{center}	\textbf{\large{Open modular functors from non-finite tensor categories}}\\	\vspace{1cm}{\large Deniz Yeral }\\ 	\vspace{5mm}{\slshape  Université Bourgogne Europe\\ CNRS\\ IMB UMR 5584\\ F-21000 Dijon\\ France}\end{center}	\vspace{0.3cm}
	\begin{abstract}\noindent 
		We show that a compact rigid balanced braided monoidal category with enough compact projective objects gives rise to a system of mapping class group representations compatible with the gluing along marked intervals. A motivation to consider non-finite tensor categories is the theory of vertex operator algebras where such categories arise as categories of modules. The mapping class group representations presented in this article admit a factorization homology description. In other words, they are of three-dimensional origin and hence obey a holographic principle. A compact projective symmetric Frobenius algebra endows the representations with a pointing that is mapping class group invariant and compatible with the gluing along intervals. This shows that, at least to some extent, many of the tools for the construction and study of spaces of conformal blocks and correlators remain available in a non-finite, but rigid setting.
	\end{abstract}
	
	\tableofcontents
	
	\section{Introduction and summary}
	The monodromy data of a conformal field theory can be described by the notion of a \emph{modular functor} \cite{Seg,MS,Tur,Til,BK}. An open-closed modular functor assigns vector spaces that are called \emph{spaces of conformal blocks} to compact oriented surfaces with parametrized intervals and circles in their boundary. The spaces of conformal blocks are endowed with an action of the mapping class group, i.e.\ the group of isotopy classes of diffeomorphisms preserving the boundary parametrizations and the orientation. These vector spaces are compatible with the gluing along intervals and circles, see e.g.\ \cite{MSWY}. An \emph{open} modular functor is the restriction to the open sector, in other words, we only allow surfaces with marked intervals. This notion is studied in \cite{MW6} as a modular algebra in the sense of \cite{GK2} over a modular operad of surfaces with marked intervals, see also \cite{LAZAROIU2001497,Moore:2006dw,Costello:2004ei} for more background on open-closed field theories.
	
	In this note, we define a class of open modular functors in a non-finite setting. We also show that this open modular functor can be given in terms of factorization homology. 
	
	More precisely, let $\cat$ be a \emph{framed $E_2$-algebra} in $\mathsf{Pr}$, where $\mathsf{Pr}$ denotes the symmetric monoidal bicategory of $k$-linear presentable categories, cocontinuous functors and linear natural isomorphisms. This means that $\cat$ is equipped with a cocontinuous monoidal product $\otimes: \cat \boxtimes \cat \to \cat$ together with a braiding and a balancing that are subject to coherence conditions. If furthermore $\mathcal{A}$ has enough compact projective objects such that its compact objects admit duals and if the duality is compatible with the balancing, then we show in Proposition \ref{prop} that $\cat$ is a $\mathsf{Pr}$-valued \emph{cyclic} framed $E_2$-algebra. 
	
	In particular, $\cat$ carries a structure of a cyclic \emph{associative} algebra. By \cite{MW6}, the cyclic associative algebra $\cat$ uniquely extends to an open modular functor that will be denoted by $\cat_{\text{\normalfont \bfseries !}}$. This means that, for any connected compact surface $\Sigma$ with at least one boundary component and with $n$ intervals in its boundary colored with objects in $X_i \in \cat$, we obtain a mapping class group  representation $\cat_{\text{\normalfont \bfseries !}}(\Sigma;X_1,\dots,X_n)$, see Figure~\ref{figure1}.

	\begin{figure}[h]
	\begin{align}\begin{array}{c}	\begin{tikzpicture}[scale=0.5]
				\begin{pgfonlayer}{nodelayer}
					\node [style=none] (0) at (-13, 3) {};
					\node [style=none] (1) at (-13, 1) {};
					\node [style=none] (2) at (-13, -1) {};
					\node [style=none] (3) at (-13, -3) {};
					\node [style=none] (4) at (-4, 1) {};
					\node [style=none] (5) at (-4, -1) {};
					\node [style=none] (6) at (-2, 0) {$\mapsto$};
					\node [style=none] (7) at (2.5, 0) {$\cat_{\text{\normalfont \bfseries !}}(\Sigma;X_1,X_2,X_3)$};
					\node [style=none] (9) at (8.5, 0) {$\mathsf{Map}(\Sigma)$};
					\node [style=none] (13) at (-13.5, 1.5) {};
					\node [style=none] (14) at (-12.5, 1.5) {};
					\node [style=none] (15) at (-12.5, 2.5) {};
					\node [style=none] (16) at (-13.5, 2.5) {};
					\node [style=none] (17) at (-8.75, 0) {};
					\node [style=none] (18) at (-8, 0) {};
					\node [style=none] (19) at (-9.25, 0.5) {};
					\node [style=none] (20) at (-7.5, 0.5) {};
					\node [style=none] (21) at (-14.25, 2) {$X_1$};
					\node [style=none] (22) at (-11.75, 1.75) {$X_2$};
					\node [style=none] (23) at (-14.5, -2) {$X_3$};
					\node [style=none] (24) at (-7.75, -2.25) {$\Sigma$};
					\node [style=none] (25) at (6.5, 0) {$\curvearrowleft$};
				\end{pgfonlayer}
				\begin{pgfonlayer}{edgelayer}
					\draw [style=black, bend left=15, looseness=0.50] (3.center) to (5.center);
					\draw [style=black, bend left=75] (2.center) to (3.center);
					\draw [style=black, bend left=75, looseness=2.25] (1.center) to (2.center);
					\draw [style=black, bend left=45] (17.center) to (18.center);
					\draw [style=black, bend right=60, looseness=1.25] (19.center) to (20.center);
					\draw [style=black, in=180, out=-15, looseness=0.25] (0.center) to (4.center);
					\draw [style=black, bend right=75, looseness=0.75] (4.center) to (5.center);
					\draw [style=black, bend left=90, looseness=0.75] (4.center) to (5.center);
					\draw [style=black, in=165, out=-75, looseness=0.75] (13.center) to (1.center);
					\draw [style=black, in=-180, out=75] (16.center) to (0.center);
					\draw [style=black, in=105, out=0] (0.center) to (15.center);
					\draw [style=black, in=-105, out=15, looseness=0.75] (1.center) to (14.center);
					\draw [style=thickblue, in=105, out=-105] (16.center) to (13.center);
					\draw [style=thickblue, in=75, out=-75, looseness=0.75] (15.center) to (14.center);
					\draw [style=thickblue, bend right=90] (2.center) to (3.center);
				\end{pgfonlayer}
			\end{tikzpicture}
		\end{array}
	\end{align}
	\caption{The value of $\cat_{\text{\normalfont \bfseries !}}$ on an open surface with three colored marked intervals.}
	\label{figure1}
	\end{figure}
	This open modular functor admits a description in terms of \emph{factorization homology}  in the following sense: Given a framed $E_2$-algebra $\cat$ in $\mathsf{Pr}$, factorization homology of surfaces~\cite{AF,BZBJ} assigns to an oriented surface $\Sigma$ a category $\int_{\Sigma} \mathcal{A} \in \mathsf{Pr}$ in a way that satisfies excision. Moreover:
	\begin{itemize}
		\item If $\Sigma$ has $n$ marked intervals in its boundary, then $\int_{\Sigma} \mathcal{A}$ becomes a module category over $\cat^{\boxtimes n}$. The action $\rhd:\cat^{\boxtimes n} \boxtimes \int_{\Sigma} \mathcal{A}  \to \int_{\Sigma} \mathcal{A} $ is given by embedding of disks into a neighborhood of the marked intervals and is equivariant with respect to the action of the mapping class group $\mathsf{Map}(\Sigma)$ on $\int_{\Sigma}\cat$.
		
		\item The canonical embedding $\emptyset \rightarrow \Sigma$ endows the category $\int_{\Sigma} \mathcal{A}$ with a pointing  $\mathcal{O}_{\Sigma} \in \int_{\Sigma} \mathcal{A}$ which is a $\mathsf{Map}(\Sigma)$-homotopy fixed point.
	\end{itemize}
	These two facts together give a $\mathsf{Map}(\Sigma)$-action on the vector space $\mathsf{Hom}_{\int_{\Sigma}\cat}(\mathcal{O}_\Sigma,X \rhd \mathcal{O}_\Sigma)$ for every $X \in \mathcal{A}^{\boxtimes n}$.
	If $\mathcal{A}$ is moreover compact rigid and has enough compact projective objects, this action is used to define an open modular functor:
	\begin{reptheorem}{thmain}
		Let $\cat$ be a framed $E_2$-algebra in $\mathsf{Pr}$ that has enough compact projective objects, that is compact rigid with a compact monoidal unit and whose balancing is compatible with duals. Let $\Sigma$ be a surface with at least one boundary component per connected component and $n \geq 1$ marked intervals in its boundary, so that each connected component of $\Sigma$ contains at least one marked interval. Let $\mathsf{FH}_{\cat}(\Sigma;-): \cat^{\boxtimes n} \to \mathsf{Vect}$ be the unique cocontinuous functor that is given by 
		\begin{align}
			X \mapsto \mathsf{Hom}_{\int_{\Sigma}\cat}(\mathcal{O}_{\Sigma},X\rhd \mathcal{O}_{\Sigma})
		\end{align} on compact projective objects $X \in \mathsf{cp}(\cat^{\boxtimes n})$. Then, the assignment $\Sigma \mapsto \mathsf{FH}_{\cat}(\Sigma;-)$ defines a $\mathsf{Pr}$-valued open modular functor.
	\end{reptheorem}
	By Corollary \ref{cor} this open modular functor coincides with $\cat_{\text{\normalfont \bfseries !}}$ and we have mapping class group equivariant isomorphisms 
	\begin{align}
		\cat_{\text{\normalfont \bfseries !}}(\Sigma;X) \cong \mathsf{Hom}_{\int_{\Sigma}\cat}(\mathcal{O}_{\Sigma},X \rhd \mathcal{O}_{\Sigma}) \quad \text{for}\quad  X \in \mathsf{cp}(\cat^{\boxtimes n})\ .
	\end{align}
	
	These results apply to non-finite categories originating from quantum groups (Example~\ref{examplehopf}) and vertex operator algebras (Example~\ref{examplevoa}), e.g.\ the $\beta\gamma$ ghosts~\cite{AW}. We also mention that by the results of \cite{BSZ}, the open modular functors constructed in this note admit an extension to a partially defined open-closed modular functor whose value on the circle is $\int_{\mathbb{S}^1}\cat$.
	
	\subsection*{Holographic principle} 
	Factorization homology admits a description in terms of skein categories \cite{Cooke,BH}. In particular, the vector space $\mathsf{Hom}_{\int_{\Sigma}\cat}(\mathcal{O}_{\Sigma},X\rhd \mathcal{O}_{\Sigma})$ can be seen as the admissible skein module, see e.g.\ \cite{CGP}, of $\Sigma \times [0,1]$ with suitable boundary labels. Similarly to \cite[Section 6]{Woi3}, we interpret this as an instance of \emph{holographic principle}, which means that the spaces of conformal blocks of $\cat_{\text{\normalfont \bfseries !}}$,  which are assigned to \emph{two-dimensional} manifolds, can be seen as \emph{three-dimensional} quantities.
	
	\subsection*{Open correlators} A modular functor describes the monodromy data of a \emph{chiral} conformal field theory. To describe the \emph{full} conformal field theory, one needs to exhibit, after the so-called holomorphic factorization, a collection of mapping class group invariant vectors inside the spaces of conformal blocks. These vectors are called \emph{correlators}, see e.g.\ \cite{algcften} for an overview. The same notion considered for the open sector of a modular functor leads to the so-called \emph{open} correlators. For the open modular functor constructed in this note, we show that compact projective symmetric Frobenius algebras $F \in \cat$ give rise to a system of open correlators, i.e.\ to a collection of vectors 
	\begin{align}
		\xi^{F}_{\Sigma} \in \cat_{\text{\normalfont \bfseries !}}(\Sigma;F,\dots,F)
	\end{align} 
	that are fixed by the action of the mapping class group of $\Sigma$ and that respect the gluing along intervals. This generalizes the results of \cite[Section 8]{Woi} to a non-finite setting.
	\begin{repcorollary}{corr}
		A compact projective symmetric Frobenius algebra $F \in \cat$ defines a consistent system of open correlators $\xi_\Sigma^F \in \cat_{\text{\normalfont \bfseries !}}(\Sigma;F,\dots,F)$. Conversely, any such system with underlying compact projective object $F$ defines a symmetric Frobenius algebra structure on $F$.
	\end{repcorollary}

	This allows to conclude the existence of open correlators for the $\beta\gamma$ ghosts (Example~\ref{excor}).

	\subsection*{Acknowledgements} I would like to thank my advisor Lukas Woike for many helpful discussions and comments related to this project. Moreover, I would like to thank to Jorge Becerra, Adrien Brochier, Lukas Müller and Peter Schauenburg for helpful discussions or comments on the draft. This project is supported by the ANR project CPJ n°ANR-22-CPJ1-0001-01 at the Institut de Mathématiques de Bourgogne. The IMB receives support from the EIPHI Graduate School (contract ANR-17-EURE-0002).

	\section{Preliminaries}
	In this section, we give a reminder on open modular functors and factorization homology. We fix an algebraically closed field $k$. 
	
	\subsection{Open modular functors} \label{subsectionone} A \emph{cyclic} operad \cite{GK1} has a structure that allows to consistently exchange the inputs and the output of the operations. A \emph{modular} operad \cite{GK2} furthermore includes self-compositions of operations.
	An open modular functor is a modular algebra over a certain modular operad, the so-called \emph{open surface operad}.  We will use a groupoid-valued model for the open surface operad $\mathsf{O}$. We refer to \cite{MW1,MW6} for the precise definitions and details.
	
	The groupoid of operations $\mathsf{O}(n)$ of total arity $n+1\geq 0$ consists of connected surfaces $\Sigma$ (always compact and oriented) with at least one boundary component and $n+1$ parametrized intervals on the boundary. Here a parametrization means an orientation-preserving embedding $[0,1]^{\sqcup (n+1)} \hookrightarrow \partial \Sigma$. The morphisms of $\mathsf{O}(n)$ are isotopy classes of orientation-preserving diffeomorphisms that also preserve boundary parametrizations. In particular, the automorphism group of $\Sigma$ is the mapping class group that will be denoted by $\mathsf{Map}(\Sigma)$. The operadic composition is given by gluing surfaces along the boundary intervals.
	
	The algebras that we consider in this note take values in $\mathsf{Pr}$, the bicategory of presentable $k$-linear  categories, cocontinuous functors and linear natural isomorphisms, see e.g.\ \cite{Brandenburg2014ReflexivityAD,BJS, BJSS} for more background. The Kelly-Deligne tensor product $\boxtimes$ endows $\mathsf{Pr}$ with a symmetric monoidal structure, the monoidal unit is the category of vector spaces $\mathsf{Vect}$. We refer to \cite{SP} for the precise definition of a symmetric monoidal bicategory.

	Cyclic and modular algebras are defined via a self-duality structure on the underlying category. Let $\cat \in \mathsf{Pr}$ be equipped with a cocontinuous pairing $\kappa: \cat \boxtimes \cat \to \mathsf{Vect}$ that is
	\begin{itemize}
		\item non-degenerate, i.e.\ there exists a copairing $\Delta: \mathsf{Vect} \to \cat \boxtimes \cat$ such that the snake relations hold up to isomorphism,
		\item symmetric, i.e.\ $\kappa$ is a $\mathbb{Z}_2$-homotopy fixed point with respect to the action of the symmetric braiding of $\mathsf{Pr}$.
	\end{itemize}
	An \emph{open modular functor} with label category $\cat$ amounts to a collection of cocontionuous functors $F(\Sigma;-):\cat^{\boxtimes n} \to \mathsf{Vect}$ together with a $\mathsf{Map}(\Sigma)$-action, where $\Sigma$ is a surface with $n$ marked intervals on its boundary. Furthermore, these functors respect the gluing in the following way: Given a surface $\Sigma'$ obtained by gluing $\Sigma$ along two intervals on $\partial \Sigma$, we have a $\mathsf{Map}(\Sigma)$-equivariant  natural isomorphism
	\begin{align}
		F(\Sigma';X_1,\dots,X_n) \cong F(\Sigma;X_1,\dots,\Delta',\dots,\Delta'',\dots,X_n) \quad \text{for}\quad  X_i \in \cat \ , 
	\end{align}
	where $\Delta = \Delta' \boxtimes \Delta''$ is inserted in the slots belonging to the intervals that we glue together. This is called \emph{excision}. Note that we use Sweedler notation for the copairing $\Delta$; it is a priori not a pure tensor.
	
	\subsection{Factorization homology}
	Inspired by \cite{BeDr}, factorization homology is introduced in \cite{Lur,AF} as a homology theory for $n$-dimensional (topological) manifolds that takes an $E_n$-algebra as input. The relevant case in this note is the factorization homology of two-dimensional manifolds, we refer to \cite{BZBJ}.
	
	Recall that a \emph{framed $E_2$-algebra} in $\mathsf{Pr}$ is a balanced braided monoidal category $\cat$ in $\mathsf{Pr}$, i.e.\ $\cat$ is equipped with
	\begin{itemize}
		\item a cocontinuous monoidal product $\otimes:\cat \boxtimes \cat \to \cat$ with a monoidal unit $I$, together with the usual associativity and unit isomorphisms that satisfy coherence conditions,
		\item a braiding, i.e.\ a natural isomorphism $c_{X,Y}:X\otimes Y\cong Y\otimes X$ that satisfies the usual hexagon relations,
		\item a balancing, i.e.\ a natural isomorphism $\theta_{X}:X\cong X$ such that 
		\begin{align} 
		\theta_{X\otimes Y} &= c_{Y,X}c_{X,Y}(\theta_X \otimes \theta_Y) \ , \\ \theta_{I}&=\text{id}_I \ . 
		\end{align}
	\end{itemize}
	Given a framed $E_2$-algebra $\cat$ and a surface $\Sigma$, one can define the factorization homology as
	\begin{align}
		\int_{\Sigma}\mathcal{A} := \underset{\substack{\varphi:\mathbb{D}^{\sqcup n}\hookrightarrow \Sigma \\ n \geq 0}}{\operatorname{hocolim}}\mathcal{A}^{\boxtimes n} \in \mathsf{Pr} \ ,
	\end{align}
	where $\mathbb{D}$ denotes the two-dimensional disk and the homotopy colimit runs over all oriented embeddings $\varphi:\mathbb{D}^{\sqcup n}\hookrightarrow \Sigma$. This assignment is functorial, meaning that for an oriented embedding $f:\Sigma \hookrightarrow \Sigma'$ we have a functor $f_*: \int_{\Sigma}\cat \to \int_{\Sigma'}\cat$ and the composition is respected up to isomorphism. In particular, the canonical embedding $\emptyset \hookrightarrow \Sigma$ induces a functor $\mathsf{Vect} \to \int_{\Sigma}\cat$ which can be identified with an object in $\mathcal{O}_{\Sigma} \in \int_{\Sigma}\cat$. In this way, the category $\int_{\Sigma}\cat$ is pointed for each surface $\Sigma$. The object $\mathcal{O}_{\Sigma}$ is called \emph{quantum structure sheaf} in \cite{BZBJ}.

	\subsection{Cyclic framed $E_2$-algebras}\label{twopointthree}
	The framed $E_2$-operad has a cyclic structure \cite{Wahl,SalWahl,Bud}. If $\cat$ is moreover a \emph{cyclic} framed $E_2$-algebra, then one can define an open modular functor in the following way: A cyclic framed $E_2$-algebra is in particular a cyclic \emph{associative} algebra. By the results in \cite{MW6}, which builds on \cite{Cos,Gia,MW1,MW2}, a cyclic associative algebra in any symmetric monoidal bicategory can be \emph{uniquely} extended to an open modular functor	that is denoted by $\cat_{\text{\normalfont \bfseries !}}$. This establishes an equivalence between cyclic associative algebras and open modular functors.
	
	In the following section, we will give sufficient conditions for a framed $E_2$-algebra $\cat$ to extend to a cyclic framed $E_2$-algebra. This will be based on a \emph{rigid} duality. Recall that an object $X$ in a monoidal category $\cat$ is called (left) dualizable if there exists an object $X^\vee$ together with maps 
	\begin{align}
		&\text{ev}_{X}:X^\vee \otimes X \to I\ , \\ &\text{coev}_{X}:I \to X \otimes X^\vee\ ,
	\end{align} 
	such that the zigzag identities are satisfied. This induces the adjunctions 
	\begin{align}\label{dualiso}
		\Hom(- \otimes X,-) &\cong \Hom(-, - \otimes X^\vee) \ , \\ \Hom(X^\vee \otimes -,-) &\cong \Hom(-, X \otimes -) \ , 
	\end{align}
	see e.g.\ \cite{EGNO,Walton}. The notion of a right dual is defined analogously. An object is called \emph{dualizable} if it admits left and right duals.

\needspace{10\baselineskip}
	\section{Factorization homology construction of non-finite open modular functors}
	In this section, we introduce a class of cyclic framed $E_2$-algebras in $\mathsf{Pr}$ and describe its associated
	open modular functor via factorization homology.
		
		\subsection{Cyclic framed $E_2$-algebras from non-finite tensor categories}
		Let $\cat\in \mathsf{Pr}$. Recall that an object $X\in \cat$ is called
		\begin{itemize}
		\item \emph{compact} if $\Hom(X,-)$ preserves filtered colimits,
		\item \emph{compact projective} if $\Hom(X,-)$ preserves arbitrary colimits.
		\end{itemize}
		We say that $\cat$ \emph{has enough compact projective objects} if every object of $\cat$ is a colimit of compact projective objects and that $\cat$ is \emph{compact rigid} if its a monoidal category whose compact objects are dualizable. In the following, we will denote by $\mathsf{c}(\cat)$ the full subcategory of compact objects, and by $\mathsf{cp}(\cat)$, the full subcategory of compact projective objects of $\cat$.
		
	    The following fact is possibly well known, but we state and prove it here in lack of a reference.
	\begin{lemma} \label{lemma}
		Suppose that $\cat \in \mathsf{Pr}$ has enough compact projective objects and is compact rigid with a compact monoidal unit. Then, the left (and right) dual of a compact projective object is also compact projective.
	\end{lemma}
	\begin{proof}
		Let $P \in \cat$ be compact projective and $P^\vee$ be its left dual. We will show that $P^\vee$ is compact projective, i.e.\ $\Hom(P^\vee, -)$ is cocontinuous. The proof for the right dual is analogous.
		\begin{itemize}
			\item By \eqref{dualiso}, we have $\Hom(P^\vee,-) \cong \Hom(I, P \otimes -).$ The monoidal unit $I$ is compact and the monoidal product $\otimes$ is cocontinuous, hence the functor $\Hom(P^\vee, -)$ preserves filtered colimits.
			\item The category $\cat$ has enough compact projective objects. In particular, it is compactly generated, i.e.\ it is the ind-completion of $\mathsf{c}(\cat)$. Consequently, the functor $\Hom(P^\vee, -)$ is cocontinuous if its restriction to $\mathsf{c}(\cat)$ preserves finite colimits. As it is additive, it preserves finite coproducts, therefore it is sufficient to prove that $\Hom(P^\vee, -) \cong \Hom(I, P \otimes -)$  preserves coequalizers. Let $C \rightrightarrows C' \rightarrow C''$ be a coequalizer in $\mathsf{c}(\cat)$. Then, $P \otimes C \rightrightarrows P \otimes C' \rightarrow P \otimes C''$ is a coequalizer in $\cat$, because $\mathsf{c}(\cat) \to \cat$ preserves finite colimits and $\otimes$ is cocontinuous. Remark that each term of this coequalizer is compact projective, this follows from \eqref{dualiso}. The category $\cat$ has enough compact projective objects and the functor $\Hom(I,-):\mathsf{cp}(\cat)\to\mathsf{Vect}$ can be extended to a cocontinuous functor $F:\cat \to \mathsf{Vect}$. In particular, $F$ preserves coequalizers. Since $F=\Hom(I,-)$ on compact projective objects, this means that $\Hom(I,P\otimes C) \rightrightarrows \Hom(I,P \otimes C') \rightarrow \Hom(I,P\otimes C'')$ is a coequalizer in $\mathsf{Vect}$. This shows that $\Hom(P^\vee, -)$ preserves coequalizers in $\mathsf{c}(\cat)$.
		\end{itemize}
	\end{proof}
	 The next proposition shows that if $\cat$ is furthermore equipped with a braiding and a balancing that is compatible with duals, i.e.\ $\theta_{X}^\vee = \theta_{X^\vee}$ for $X$ dualizable, then $\cat$ is a cyclic framed $E_2$-algebra in $\mathsf{Pr}$.

	\begin{proposition} \label{prop}
		Let $\cat$ be a framed $E_2$-algebra in $\mathsf{Pr}$ that has enough compact projective objects, that is compact rigid with a compact monoidal unit and whose balancing is compatible with duals. Then $\mathcal{A}$ carries a \emph{cyclic} framed $E_2$-algebra structure in $\mathsf{Pr}$.
	\end{proposition}

	\begin{proof}
		We will show that $\mathcal{A}$ is a self-dual balanced braided algebra in the sense of \cite[Definition 5.4]{MW1}. More precisely, we will define a non-degenerate pairing $\kappa$ on $\mathcal{A}$ together with a natural isomorphism $\gamma_{X,Y}: \kappa(X,Y) \cong \kappa(I,X \otimes Y)$ that satisfies the conditions (U) and (PB) in \cite[Definition 5.4]{MW1}. This concludes the proof because by \cite[Theorem 5.13]{MW1}, self-dual balanced braided algebras and cyclic framed $E_2$-algebras are equivalent.
		\begin{itemize}
			\item Let us define $\kappa:\mathcal{A} \boxtimes \mathcal{A} \to \mathsf{Vect}$ to be the unique cocontinuous functor that is given by 
			\begin{align}
				(P,Q) \mapsto \mathsf{Hom}_{\mathcal{A}}(I,P\otimes Q)
			\end{align} 
			on compact projective objects $P,Q \in \mathsf{cp}(\cat)$. We will prove that $\kappa$ is non-degenerate and the corresponding copairing is given by the coend \begin{align} \Delta = \int^{Q \in \mathsf{cp}(\cat)}Q^\vee \boxtimes Q \in \mathcal{A}\boxtimes \mathcal{A} \ . 
				\end{align} In other words, we will exhibit an isomorphism $\kappa(-,\Delta')\otimes \Delta'' \cong \text{id}_{\mathcal{A}}$. Indeed, both of these functors being cocontinuous, it is enough to show that 
				\begin{align} 
					\kappa(P,\Delta')\otimes \Delta'' \cong P
				\end{align} 
				for $P$ compact projective. The functor $\kappa(P,-)\otimes -$ is cocontinuous, hence we can pull the coend out and we have
			\begin{align}
				\kappa(P,\Delta')\otimes \Delta'' \cong \int^{Q \in \mathsf{cp}(\cat)}\kappa(P,Q^\vee)\otimes Q \stackrel{\text{Lemma } \ref{lemma}}{\cong}
				\int^{Q \in \mathsf{cp}(\cat)}\mathsf{Hom}_{\mathcal{A}}(I,P\otimes Q^\vee)\otimes Q \\ \stackrel{\eqref{dualiso}}{\cong} \int^{Q \in \mathsf{cp}(\cat)}\mathsf{Hom}_{\mathcal{A}}(Q,P)\otimes Q \cong P \ ,
			\end{align}
		where in the last step we used the Yoneda Lemma. The isomorphism 
		\begin{align} 
			\Delta' \otimes \kappa(\Delta'',-) \cong \text{id}_{\mathcal{A}}
		\end{align} 
		is obtained analogously.
		\item Let $\varepsilon:\mathcal{A}\to \mathsf{Vect}$ be the unique cocontinuous functor that is given by $P \mapsto \mathsf{Hom}_{\mathcal{A}}(I,P)$ on compact projective objects $P$. Then, we have 
		\begin{align} 
			\varepsilon(P\otimes Q) \cong \Hom(I,P\otimes Q)
		\end{align} 
		for $P,Q \in \mathsf{cp}(\cat)$ because $P \otimes Q$ is compact projective by \eqref{dualiso}. This means that $\kappa \cong \varepsilon \circ \otimes$ since both of these functors are cocontinuous and they agree on compact projective objects. Then, the unit isomorphisms of the monoidal product defines the desired natural isomorphism $\gamma_{X,Y}: \kappa(X,Y) \cong \kappa(I,X \otimes Y)$ and the condition (U) in \cite[Definition 5.4]{MW1} is automatically satisfied.
		\item To finish the proof, we need to show that the condition (PB) in \cite[Definition 5.4]{MW1} holds, i.e. the isomorphisms $\kappa(X,Y) \cong \kappa(X,Y)$ induced by $\theta_X \otimes \text{id}_Y$ and $\text{id}_X \otimes \theta_Y$ coincide, where $\theta$ denotes the balancing. Again, it is sufficient to prove this for compact projective objects $P,Q$. We have
		\begin{align}
			\kappa(P,Q) \cong \mathsf{Hom}_{\mathcal{A}}(I,P \otimes Q) \cong \mathsf{Hom}_{\mathcal{A}}(P^\vee,Q) \ .
		\end{align}
		The automorphism on $\mathsf{Hom}_{\mathcal{A}}(P^\vee,Q)$ that is induced by $\theta_P \otimes \text{id}_Q$ is the assignment $f \mapsto f \circ (\theta_{P})^\vee$, whereas $\text{id}_P \otimes \theta_Q$ induces $f \mapsto \theta_Q \circ f$. As the balancing is compatible with the duals, $(\theta_{P})^\vee = \theta_{P^\vee}$, and as $\theta$ is a natural isomorphism, these two assignments are the same. This concludes the point (PB).
		\end{itemize}
	\end{proof}
	
	\begin{exxample}\label{exampleabelian}
		Let $\mathcal{C}$ be a linear abelian ribbon category. In particular, $\mathcal{C}$ is finitely cocomplete hence is a framed $E_2$-algebra in $\mathsf{Rex}$, the symmetric monoidal bicategory of finitely cocomplete linear categories, right exact functors and linear natural isomorphisms. Its ind completion $\mathsf{ind}\, \mathcal{C}$ is a framed $E_2$-algebra in $\mathsf{Pr}_\mathsf{c}$, the subcategory of $\mathsf{Pr}$ that consists of compactly generated presentable linear categories, cocontinuous and compact preserving functors and linear natural isomorphisms \cite[Section~3.1]{BZBJ}. If furthermore $\mathcal{C}$ has enough projective objects, then $\mathsf{ind}\, \mathcal{C}$ has enough compact projective objects. By Proposition~\ref{prop}, $\mathsf{ind}\, \mathcal{C}$ is therefore a cyclic framed $E_2$-algebra in $\mathsf{Pr}$. As a technical subtlety, we should highlight that $\mathcal{C}$, however, is generally \emph{not} a cyclic framed $E_2$-algebra in $\mathsf{Rex}$ or $\mathsf{Pr}_\mathsf{c}$ because the copairing $\Delta$ might only exist as an ind object, see also \cite[Remark~2.11]{BJS} for a related problem. For this reason, passing to the ind completion is not an artificial step, but necessary to obtain the cyclic structure.	We will discuss more specific examples falling into this example class in Example~\ref{examplehopf} and \ref{examplevoa}.
	\end{exxample}

	\subsection{The open modular functor}
	As we recalled in Subsection \ref{twopointthree}, the cyclic framed $E_2$-algebra $\cat$ is in particular a cyclic associative algebra, which in this case is given by forgetting the balancing and the braiding. The cyclic associative algebra extends to an open modular functor that we denote by $\cat_{\text{\normalfont \bfseries !}}$ \cite{MW6}. 
	
	In this subsection, we will show that this open modular functor can be described via factorization homology. Recall that given a surface $\Sigma$ with $n$ marked intervals, inclusion of a disk $\mathbb{D} \cong [0,1] \times [0,1]$ along an interval defines an action 
	\begin{align} 
		\rhd:\cat^{\boxtimes n} \boxtimes \int_{\Sigma}\cat \to \int_{\Sigma}\cat \ , 
	\end{align} 
	see e.g.\ \cite[Section 2]{BZBJ}. This action can be used to construct an open modular functor. The following generalizes \cite[Theorem 6.1]{Woi3}:
	
	\begin{theorem} \label{thmain}
		Let $\cat$ be a framed $E_2$-algebra in $\mathsf{Pr}$ that has enough compact projective objects, that is compact rigid with a compact monoidal unit and whose balancing is compatible with duals. Let $\Sigma$ be a surface with at least one boundary component per connected component and $n \geq 1$ marked intervals in its boundary, so that each connected component of $\Sigma$ contains at least one marked interval. Let $\mathsf{FH}_{\cat}(\Sigma;-): \cat^{\boxtimes n} \to \mathsf{Vect}$ be the unique cocontinuous functor that is given by 
		\begin{align}
			X \mapsto \mathsf{Hom}_{\int_{\Sigma}\cat}(\mathcal{O}_{\Sigma},X\rhd \mathcal{O}_{\Sigma})
		\end{align} on compact projective objects $X \in \mathsf{cp}(\cat^{\boxtimes n})$. Then, the assignment $\Sigma \mapsto \mathsf{FH}_{\cat}(\Sigma;-)$ defines a $\mathsf{Pr}$-valued open modular functor.
	\end{theorem}

	\begin{proof}
	 For any $X \in \mathsf{cp}(\cat^{\boxtimes n})$, the object $X \rhd \mathcal{O}_{\Sigma}$ has a homotopy fixed point structure with respect to the mapping class group because the mapping class group acts on $\int_{\Sigma}\cat$ as an $\cat^{\boxtimes n}$-module map. This induces a mapping class group action on $\mathsf{Hom}_{\int_{\Sigma}\cat}(\mathcal{O}_{\Sigma},X \rhd \mathcal{O}_{\Sigma})$ and hence on $\mathsf{FH}_{\cat}(\Sigma;-)$.
	 
	 Let us show the excision for the spaces of conformal blocks: Let $\Sigma'$ be a surface with $n \geq 1$ marked interval that is obtained by gluing a surface $\Sigma$ along two intervals on $\partial \Sigma$. We will show that
	 \begin{align}
	 	\mathsf{FH}_\cat(\Sigma';X) \cong \mathsf{FH}_\cat(\Sigma;X\boxtimes \Delta) 
	 	\quad \text{for}\quad
	 	 X \in \cat^{\boxtimes n} \ , \label{excisionone}
	 \end{align}
	 where $\Delta$ is given by the coend $\int^{P \in \mathsf{cp}(\cat)}P^\vee \boxtimes P$. As $\mathsf{FH}_\cat(\Sigma;-)$ is cocontinuous, we can pull the coend out, hence \eqref{excisionone} is equivalent to 
	 \begin{align}
	 	\mathsf{FH}_\cat(\Sigma';X) \cong \int^{P \in \mathsf{cp}(\cat)}\mathsf{FH}_\cat(\Sigma;X\boxtimes P^\vee \boxtimes P ) \quad \text{for}\quad X \in \cat^{\boxtimes n} \label{excisiontwo} \ .
	 \end{align}
	 Since we have cocontinuous functors in both sides, it is enough to show \eqref{excisiontwo} on compact projective objects in $\cat^{\boxtimes n}$, which leads to
	 \begin{align}
	 	\mathsf{Hom}_{\int_{\Sigma'}\cat}(\mathcal{O}_{\Sigma'}^\cat, X \rhd \mathcal{O}_{\Sigma'}^\cat) \cong \int^{P \in \mathsf{cp}(\cat)}\mathsf{Hom}_{\int_{\Sigma}\cat}(\mathcal{O}_{\Sigma}^\cat, (X \boxtimes P^\vee \boxtimes P) \rhd \mathcal{O}_{\Sigma}^\cat) \ , 
	 \end{align}
	 where $ X\in \mathsf{cp}(\cat^{\boxtimes n})$. By the same argument as for \eqref{dualiso}, this is equivalent to
	 \begin{align}
	 	\mathsf{Hom}_{\int_{\Sigma'}\cat}(X^\vee \rhd\mathcal{O}_{\Sigma'}^\cat, \mathcal{O}_{\Sigma'}^\cat) &\cong \int^{P \in \mathsf{cp}(\cat)}\mathsf{Hom}_{\int_{\Sigma}\cat}((X^\vee \boxtimes P^{\vee \vee} \boxtimes P^\vee) \rhd \mathcal{O}_{\Sigma}^\cat,  \mathcal{O}_{\Sigma}^\cat) \\ &\cong \int^{P \in \mathsf{cp}(\cat)}\mathsf{Hom}_{\int_{\Sigma}\cat}((X^\vee \boxtimes P^\vee \boxtimes P) \rhd \mathcal{O}_{\Sigma}^\cat,  \mathcal{O}_{\Sigma}^\cat) \ . \label{excisionthree}
	 \end{align}
	 To show \eqref{excisionthree}, we will use the description of factorization homology in terms of admissible skein modules. Roughly, an object of the skein category $\mathsf{skcat}_{\mathsf{cp}(\cat)}(\Sigma)$ is an embedding of disks into $\Sigma$ that are labeled by compact projective objects of $\cat$, such that each connected component of $\Sigma$ contains at least one embedded disk. The morphism space between two such configurations $X$ and $Y$ is given by the admissible skein module of $\Sigma \times [0,1]$ with the label $X$ on the boundary component $\Sigma \times \{0\}$ and $Y$ on $\Sigma \times \{1\}$; we denote this vector space by $\mathsf{sk}_{\mathsf{cp}(\cat)}(\Sigma \times [0,1],X,Y)$. We refer to \cite{CGP,BH,RST} for the details. 
	 
	 By \cite[Corollary 3.12]{BH}, we have \begin{align} \int_{\Sigma'}\cat \simeq \mathsf{Fun}(\mathsf{skcat}_{\mathsf{cp}(\cat)}(\Sigma')^{\text{op}},\mathsf{Vect}) \ , 
	 \end{align} here the right hand side is the presheaf category on $\mathsf{skcat}_{\mathsf{cp}(\cat)}(\Sigma')$. Under this equivalence, the distinguished object $\mathcal{O}_{\Sigma'}$ is the presheaf $\mathsf{sk}_{\mathsf{cp}(\cat)}(\Sigma' \times [0,1],-,\emptyset)$, see \cite[Remark 2.10]{BH}.

	 The object $X^\vee \rhd \mathcal{O}_{\Sigma'}$ can be seen as an object in $\mathsf{skcat}_{\mathsf{cp}(\cat)}(\Sigma')$ that we will again denote by $X^\vee$. Here it is important that every connected component of $\Sigma'$ contains a marked interval. In the same way, $X^\vee \boxtimes P^\vee \boxtimes P$ can be seen as an object in $\mathsf{skcat}_{\mathsf{cp}(\cat)}(\Sigma)$ that we will again denote by $X^\vee \boxtimes P^\vee \boxtimes P$. Now by Yoneda Lemma, \eqref{excisionthree} is equivalent to 
	 \begin{align}
	 	\mathsf{sk}_{\mathsf{cp}(\cat)}(\Sigma' \times [0,1],X^\vee,\emptyset) \cong \int^{P \in \mathsf{cp}(\cat)}\mathsf{sk}_{\mathsf{cp}(\cat)}(\Sigma \times [0,1],X^\vee \boxtimes P^\vee \boxtimes P,\emptyset) \ . \label{excisionfour}
	 \end{align}
	 We will proceed by showing that \eqref{excisionfour} is a consequence of excision for admissible skein modules. First, we observe that $\Sigma' \times [0,1]$ can be obtained by gluing $\Sigma \times [0,1]$ along two copies of $[0,1] \times [0,1]$ that are given by the marked intervals. Topologically, $[0,1] \times [0,1]$ is just a disk and $\mathsf{sk}_{\mathsf{cp}(\cat)}(\mathbb{D})\simeq \mathsf{cp}(\cat)$. Hence the excision theorem \cite[Theorem 3.1]{RST} tells us that the left hand side of \eqref{excisionfour} can be written as a coend over $P \in \mathsf{cp}(\cat)$ of the skein module of $\Sigma \times [0,1]$ that has the boundary label $X$ and also $P$ and $P^\vee$ inserted on the disks that we glue to get $\Sigma' \times [0,1]$. This skein module is isomorphic (as a $\mathsf{Map}(\Sigma)$-representation) to $\mathsf{sk}_{\mathsf{cp}(\cat)}(\Sigma \times [0,1],X^\vee \boxtimes P^\vee \boxtimes P,\emptyset)$. This can be seen by choosing a bigger disk on the boundary of $\Sigma \times [0,1]$ that contains both the disk according to which $\Sigma \times [0,1]$ is glued and the disk that represents the $\cat$-action on $\Sigma$, see Figure~\ref{figure2}. Consequently, \eqref{excisionfour} holds. Remark that by \cite[Lemma 2.7]{RST} the corner structure of $\Sigma \times [0,1]$ is not relevant.
	 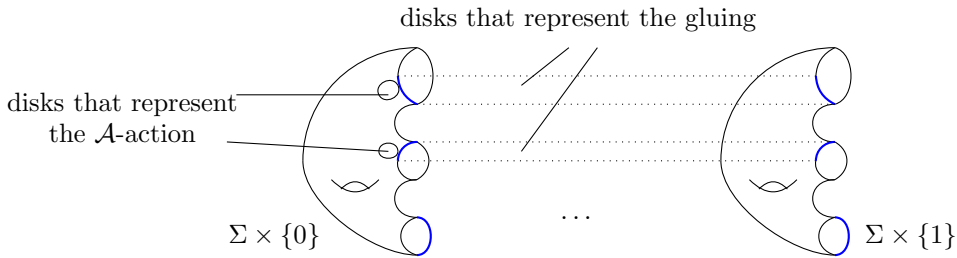
\begin{figure}[h]
	 \begin{align}\begin{array}{c}	\begin{tikzpicture}[scale=0.5]
	 			\begin{pgfonlayer}{nodelayer}
	 				\node [style=none] (0) at (-1.25, 3) {};
	 				\node [style=none] (1) at (-1.25, 1.5) {};
	 				\node [style=none] (2) at (-1.25, 0.5) {};
	 				\node [style=none] (3) at (-1.25, -0.5) {};
	 				\node [style=none] (4) at (-1.25, -1.5) {};
	 				\node [style=none] (5) at (-1.25, -2.5) {};
	 				\node [style=none] (6) at (-4.25, 0) {};
	 				\node [style=none] (7) at (-1.75, 2.25) {};
	 				\node [style=none] (8) at (-1.75, 0) {};
	 				\node [style=none] (9) at (-2.5, -0.75) {};
	 				\node [style=none] (10) at (-3.25, -0.75) {};
	 				\node [style=none] (11) at (-2.25, -0.5) {};
	 				\node [style=none] (12) at (-3.5, -0.5) {};
	 				\node [style=none] (26) at (-1.75, 0.25) {};
	 				\node [style=none] (27) at (-2.25, 0.25) {};
	 				\node [style=none] (28) at (-1.75, 2) {};
	 				\node [style=none] (29) at (-2.25, 1.75) {};
	 				\node [style=none] (30) at (-2, 1.75) {};
	 				\node [style=none] (31) at (-2, 0.25) {};
	 				\node [style=none] (32) at (-6, 1.75) {};
	 				\node [style=none] (33) at (-6.25, 0.5) {};
	 				\node [style=none] (34) at (-9, 1.5) {disks that represent};
	 				\node [style=none] (35) at (1.5, 2) {};
	 				\node [style=none] (36) at (1.5, 0.25) {};
	 				\node [style=none] (37) at (3.5, 3) {};
	 				\node [style=none] (38) at (1.75, 3.5) {};
	 				\node [style=none] (39) at (2.75, 3) {};
	 				\node [style=none] (40) at (3, 3.75) {disks that represent the gluing};
	 				\node [style=none] (41) at (3, -1.5) {\dots};
	 				\node [style=none] (42) at (-9, 0.75) {the $\cat$-action};
	 				\node [style=none] (43) at (-5, -2) {$\Sigma \times \{ 0 \}$};
	 				\node [style=none] (44) at (11.75, -2) {$\Sigma \times \{ 1 \} $};
	 				\node [style=none] (45) at (12, 5) {};
	 				\node [style=none] (46) at (9.75, 3) {};
	 				\node [style=none] (47) at (9.75, 1.5) {};
	 				\node [style=none] (48) at (9.75, 0.5) {};
	 				\node [style=none] (49) at (9.75, -0.5) {};
	 				\node [style=none] (50) at (9.75, -1.5) {};
	 				\node [style=none] (51) at (9.75, -2.5) {};
	 				\node [style=none] (52) at (6.75, 0) {};
	 				\node [style=none] (53) at (9.25, 2.25) {};
	 				\node [style=none] (54) at (9.25, 0) {};
	 				\node [style=none] (55) at (8.5, -0.75) {};
	 				\node [style=none] (56) at (7.75, -0.75) {};
	 				\node [style=none] (57) at (8.75, -0.5) {};
	 				\node [style=none] (58) at (7.5, -0.5) {};
	 				\node [style=none] (59) at (9.25, 0.25) {};
	 				\node [style=none] (61) at (9.25, 2) {};
	 				\node [style=none] (66) at (6.75, 0.25) {};
	 			\end{pgfonlayer}
	 			\begin{pgfonlayer}{edgelayer}
	 				\draw [bend left=75] (0.center) to (1.center);
	 				\draw [bend left=60, looseness=1.25] (2.center) to (3.center);
	 				\draw [bend right=75, looseness=1.50] (4.center) to (5.center);
	 				\draw [in=-90, out=180, looseness=0.75] (5.center) to (6.center);
	 				\draw [in=90, out=180] (0.center) to (6.center);
	 				\draw [bend right=90, looseness=2.00] (1.center) to (2.center);
	 				\draw [bend right=90, looseness=2.00] (3.center) to (4.center);
	 				\draw [style=thickblue, bend left=90, looseness=1.25] (4.center) to (5.center);
	 				\draw [style=thickblue, bend left] (1.center) to (7.center);
	 				\draw [style=thickblue, bend right=45] (2.center) to (8.center);
	 				\draw [bend left=60] (3.center) to (8.center);
	 				\draw [bend right] (0.center) to (7.center);
	 				\draw [bend left=45, looseness=1.25] (11.center) to (12.center);
	 				\draw [bend right=45, looseness=1.25] (9.center) to (10.center);
	 				\draw [bend right=90, looseness=1.50] (28.center) to (29.center);
	 				\draw [bend left=90, looseness=1.50] (28.center) to (29.center);
	 				\draw [bend right=90, looseness=1.50] (26.center) to (27.center);
	 				\draw [bend left=90, looseness=1.25] (26.center) to (27.center);
	 				\draw (32.center) to (30.center);
	 				\draw (33.center) to (31.center);
	 				\draw (39.center) to (35.center);
	 				\draw (36.center) to (37.center);
	 				\draw [bend left=75] (46.center) to (47.center);
	 				\draw [bend left=60, looseness=1.25] (48.center) to (49.center);
	 				\draw [bend right=75, looseness=1.50] (50.center) to (51.center);
	 				\draw [in=-90, out=180, looseness=0.75] (51.center) to (52.center);
	 				\draw [in=90, out=180] (46.center) to (52.center);
	 				\draw [bend right=90, looseness=2.00] (47.center) to (48.center);
	 				\draw [bend right=90, looseness=2.00] (49.center) to (50.center);
	 				\draw [style=thickblue, bend left=90, looseness=1.25] (50.center) to (51.center);
	 				\draw [style=thickblue, bend left] (47.center) to (53.center);
	 				\draw [style=thickblue, bend right=45] (48.center) to (54.center);
	 				\draw [bend left=60] (49.center) to (54.center);
	 				\draw [bend right] (46.center) to (53.center);
	 				\draw [bend left=45, looseness=1.25] (57.center) to (58.center);
	 				\draw [bend right=45, looseness=1.25] (55.center) to (56.center);
	 				\draw [style=dashed] (2.center) to (48.center);
	 				\draw [style=dashed] (54.center) to (8.center);
	 				\draw [style=dashed] (1.center) to (47.center);
	 				\draw [style=dashed] (53.center) to (7.center);
	 			\end{pgfonlayer}
	 		\end{tikzpicture}
	 	\end{array}
	 \end{align}
	 \caption{A picture of $\Sigma \times [0,1]$.}\label{figure2}
	\end{figure}
	
	 We still cannot conclude that $\mathsf{FH}_{\cat}$ is an open modular functor in the sense of Subsection \ref{subsectionone}, because it is not defined on surfaces \emph{without} marked intervals (but at least one boundary component per connected component). This is not an issue, because we can uniquely extend this construction to surfaces without intervals: First, consider the restriction of $\mathsf{FH}_{\cat}$ on disks with marked intervals. This defines a cyclic associative algebra. The value on the disk without intervals is given by inserting the unit on a disk with an interval. Then, the modular extension construction in \cite{MW1} gives an open modular functor whose value on surfaces with marked intervals is given by $\mathsf{FH}_{\cat}$; this is because $\mathsf{FH}_{\cat}$ satisfies the excision for surfaces with marked intervals. Consequently, by insertion of the unit, $\mathsf{FH}_{\cat}$ can also be defined on surfaces without marked intervals and hence is an open modular functor.
	\end{proof}

	\begin{corollary} \label{cor}
		Let $\cat$ be a framed $E_2$-algebra in $\mathsf{Pr}$ that has enough compact projective objects, that is compact rigid with a compact monoidal unit and whose balancing is compatible with duals. Then for each surface $\Sigma$ with at least one boundary component per connected component and $n \geq 1$ marked interval in its boundary, so that each connected component of $\Sigma$ contains at least one marked interval, 
		the unique open modular functor $\cat_{\text{\normalfont \bfseries !}}$ extending $\cat$ is characterized through a canonical isomorphism
		\begin{align}
			\cat_{\text{\normalfont \bfseries !}}(\Sigma;X) \cong \mathsf{Hom}_{\int_{\Sigma}\cat}(\mathcal{O}_{\Sigma},X \rhd \mathcal{O}_{\Sigma})
			\quad \text{for}\quad X \in \mathsf{cp}(\cat^{\boxtimes n}) \label{corollary}
		\end{align}
		that is mapping class group equivariant and that respects the gluing along intervals.
	\end{corollary}
	\begin{proof}
		We will show that $\mathcal{A}_{\text{\normalfont \bfseries !}}$ and $\mathsf{FH}_\cat$ are equivalent open modular functors. This is sufficient because the right hand side in \eqref{corollary} is the value of $\mathsf{FH}_\cat(\Sigma;-)$ on compact projective objects. By \cite{MW6}, it is enough to compare $\mathcal{\cat}_{!}$ and $\mathsf{FH}_\cat$ at disks with marked intervals. This is immediate, since we have
		\begin{align}
			\mathsf{Hom}_{\int_{\mathbb{D}}\cat}(\mathcal{O}_{\mathbb{D}},(X_1 \boxtimes \dots \boxtimes X_n) \rhd \mathcal{O}_{\mathbb{D}}) \cong \Hom(I,X_1 \otimes \dots \otimes X_n) \quad \text{for}\quad X_i \in \mathsf{cp}(\cat) \ .
		\end{align}
	\end{proof}
	
	\begin{exxample}
		The cylinder $\mathbb{S}^1 \times [0,1]$ without marked intervals can be obtained by gluing a disk along two marked intervals. Then, the excision for $\cat_{\text{\normalfont \bfseries !}}$ tells us
		\begin{align}
			\cat_{\text{\normalfont \bfseries !}}(\mathbb{S}^1 \times [0,1]) \cong 
			 \kappa(\Delta)\cong \int^{P \in \mathsf{cp}(\mathcal{A})} \Hom(I,P^\vee \otimes P)\cong \int^{P \in \mathsf{cp}(\mathcal{A})} \Hom(P, P) \ . 
			\end{align}
			The action of the mapping class group $\mathsf{Map}(\mathbb{S}^1 \times [0,1])\cong \mathbb{Z}_2$ is induced by the map sending a morphism $f:P\to P$ to its dual. 
			If $\mathcal{A}$ is semisimple and has a set of simple objects indexed by $J$, this vector space is isomorphic to $|J|$ copies
			of the ground field \cite[Remark~4.19]{BZBJ}.
			Therefore, the mapping class group representations in Corollary~\ref{cor} are generally infinite-dimensional.
		\end{exxample}

\begin{exxample}[Examples from Hopf algebras]\label{examplehopf}
	A class of examples of categories that satisfy the conditions in Proposition \ref{prop} comes from comodule categories. More precisely, the category of comodules of a semiperfect Hopf algebra has enough compact projective objects. This also includes the case of a cosemisimple Hopf algebra. In particular, the category of comodules of the restricted dual of $U_q(\mathfrak{g})$, where $\mathfrak{g}$ is a finite-dimensional simple Lie algebra and $q$ is not a root of unity, is a cyclic framed $E_2$-algebra in $\mathsf{Pr}$. The resulting category is semisimple with infinitely many simple
	objects and the value of the associated open modular functor on a surface of genus $g$ with one puncture is given in terms of $g$ tensor copies of the elliptic double from \cite{BJ}, see \cite[Corollary 6.11]{BZBJ}. 
\end{exxample}

\begin{exxample}[Examples from vertex operator algebras]\label{examplevoa}
	The module category $\mathscr{F}$ for the rank 1 bosonic ghost vertex operator algebra (the so-called \emph{$\beta\gamma$ ghosts}) considered in \cite{AW} is a linear abelian ribbon category with enough projective objects, but is neither rational nor finite (the $\beta\gamma$ ghosts are not $C_2$-cofinite). With Example~\ref{exampleabelian},
	we conclude that Theorem~\ref{thmain} applies to the ind completion $\mathbf{F}:=\mathsf{ind}\,\mathscr{F}$ of $\mathscr{F}$.
	In particular, we obtain mapping class group representations $\mathbf{F}_{\text{\normalfont \bfseries !}}(\Sigma)$ for all connected compact oriented surfaces with at least one boundary component, even without rationality or finiteness.
	For the annulus and the punctured torus, respectively, we find using excision
\begin{align}	\mathbf{F}_{\text{\normalfont \bfseries !}}(\mathbb{S}^1 \times [0,1]) &\cong \int^{P\in\mathsf{Proj}\, \mathscr{F}} \mathsf{Hom}_{\mathscr{F}}(P,P)=HH_0(\mathscr{F}) \ , \\
	\mathbf{F}_{\text{\normalfont \bfseries !}}(\mathbb{T}^2 \setminus  \mathbb{D}^2) &\cong  \int^{P,Q\in\mathsf{Proj}\, \mathscr{F}} \mathsf{Hom}_{\mathscr{F}}(Q\otimes P,P\otimes Q) \ ,
	\end{align}
	where $\mathsf{Proj}\, \mathscr{F}$ denotes the full subcategory of projective objects and $HH_0(\mathscr{F})$ denotes the zeroth Hochschild homology of $\mathscr{F}$.
	By Corollary~\ref{cor}, the open spaces of conformal blocks
	$\mathbf{F}_{\text{\normalfont \bfseries !}}(\Sigma)$ admit a holographic description in terms of factorization homology, or equivalently, admissible skeins. 
	The $\beta\gamma$ ghosts are not the only example of a module category of a vertex operator algebra that is linear abelian rigid with enough projective objects, see e.g.\ the list in the introduction of \cite{Creutzig:2024qvb}.
	\end{exxample}

\subsection{Skein theory and ansular functors}
	The cyclic framed $E_2$-algebra constructed in Proposition \ref{prop} extends to a system of handlebody group representations compatible with the gluing along boundary disks, i.e.\ to a so-called \emph{ansular functor} \cite{MW2}. More precisely, for any handlebody with $n$ parametrized disks on the boundary, we have a cocontinuous functor $\widehat{\cat}(H):\cat^{\boxtimes n} \to \mathsf{Vect}$, and these functors satisfy excision. The ansular functor $\widehat{\cat}$ is related to $\cat_{\text{\normalfont \bfseries !}}$ in the following sense: Let $\Sigma$ be a connected surface with $n$ boundary components and let $H$ be a handlebody with $n$ disks embedded in its boundary such that $\partial H = \Sigma$, where $\partial H$ is the surface obtained by removing the interiors of the embedded disks on the boundary. By \cite[Section 4]{BW}, we have a functor
	\begin{align}
		\Phi_{\mathcal{A}}(H):\int_{\Sigma}\cat \to \cat^{\boxtimes n}\label{phimap}
	\end{align}
	called \emph{handlebody skein module} of $H$, the terminology coincides with the usual skein-theoretic constructions in the finite rigid case as proved in \cite{MW3}. The value of $\Phi_{\mathcal{A}}(H)$ on $\mathcal{O}_{\Sigma}$ is $\widehat{\cat}(H)$, here $\widehat{\mathcal{A}}(H)$ is seen as an object in $\mathcal{A}^{\boxtimes n}$ by cyclicity. Now if we choose an interval on each boundary component of $\Sigma$ (for the resulting operation in the open surface operad, we use the same symbol), we obtain the map
	\begin{align}
		\mathsf{Add}_{\Sigma,H}: \cat_{\text{\normalfont \bfseries !}}(\Sigma;P) \cong \mathsf{Hom}_{\int_{\Sigma}\cat}(\mathcal{O}_{\Sigma},P\rhd \mathcal{O}_{\Sigma}) \xrightarrow{\Phi_{\cat}(H)\circ -} \mathsf{Hom}_{\cat^{\boxtimes n}}(\widehat{\cat}(H), P\otimes\widehat{\cat} (H))   \label{eqskeinaction}
	\end{align}
	for $P \in \mathsf{cp(\cat^{\boxtimes n})}$.
	This map is an isomorphism if $\mathcal{A}$ is the ind completion of a modular category \cite[Corollary 6.6]{Woi2}. If $\cat$ is in addition a \emph{connected} cyclic framed $E_2$-algebra in the sense of \cite[Definition 5.8]{BW}, then it extends to a modular functor. Then the arguments of \cite[Theorem 6.1]{Woi2} show that the map $\mathsf{Add}_{\Sigma,H}$ is moreover $\mathsf{Map}(\Sigma)$-equivariant.
	
	The map $\mathsf{Add}_{\Sigma,H}$ is given by the action of the skein algebra of $\Sigma$ on the handlebody skein modules in the following sense: Similarly to \cite{MW3}, the ansular functor $\widehat{\cat}$ can be described by skein modules of handlebodies. In particular, the value of $\widehat{\cat}(H)$ on a compact projective object $X \in \mathsf{cp}(\cat^{\boxtimes n})$ coincides with the skein module $\mathsf{sk}_{\mathsf{cp}(\cat)}(H;X)$ of $H$ with the boundary parametrization given by $n$ embedded disks colored by $X$.

\subsection{Evaluation on modules over the reflection equation algebra} \label{subsecthreefour}
	Let $\Sigma$ be a connected surface with $n\ge 1$ parametrized boundary components. We add an interval around the image of a fixed point on $\mathbb{S}^1$ under the parametrization. For the cyclic framed $E_2$-algebra from Proposition \ref{prop}, the functor $\cat_{\text{\normalfont \bfseries !}}(\Sigma;-):\mathcal{A}^{\boxtimes n} \to\mathsf{Vect}$ extends to a functor
	\begin{align}
		\Phi_\mathcal{A} (\Sigma) :	\left(\int_{\mathbb{S}^1 \times [0,1]} \mathcal{A} \right)^{\boxtimes n} \to \mathsf{Vect}   \label{phimap2d}
	\end{align}
	whose value on $\mathcal{O}_{\mathbb{S}^1 \times [0,1]}^{\boxtimes n}$ agrees with $\cat_{\text{\normalfont \bfseries !}}(\Sigma;I^{\boxtimes n})$, see the proof of \cite[Theorem 5.1]{MW6} (the construction is analogous to the one of \eqref{phimap}). By \cite[Theorem~5.14]{BZBJ}, the factorization homology $\int_{\mathbb{S}^1 \times [0,1]} \mathcal{A}$ is equivalent to the category of right modules in $\mathcal{A}$ over the \emph{reflection equation algebra} $\mathbb{F}=\int^{P\in\mathsf{cp}(\cat)} P^\vee \otimes P$ \cite[Definition 4.17]{BZBJ}. In other words, \eqref{phimap2d} describes the extension of $\cat_{\text{\normalfont \bfseries !}}(\Sigma;-)$ from objects in $\cat$ to modules over the reflection equation algebra such that
	\begin{align}
		\Phi_\mathcal{A}(\Sigma) (X_1 \otimes \mathbb{F},\dots,X_n\otimes \mathbb{F}) \cong \cat_{\text{\normalfont \bfseries !}}(\Sigma;  X_1 ,\dots,X_n      ) \ . 
	\end{align}
	In particular, $\cat_{\text{\normalfont \bfseries !}}(\Sigma;I^{\boxtimes n})$ is a module over $n$ tensor copies of the endomorphism algebra of $\mathcal{O}_{\mathbb{S}^1 \times [0,1]}$ or, equivalently, the endomorphism algebra of $\mathbb{F}$ as $\mathbb{F}$-module, which is the algebra  $\Hom(I,\mathbb{F})$ of invariants of the reflection equation algebra. 
 
\subsection{Connection to the Drinfeld center}

As we recalled in the previous subsection, the open modular functor $\cat_{\text{\normalfont \bfseries !}}$ constructed in this note induces a map on the factorization homology $\int_{\mathbb{S}^1 \times [0,1]}\cat$ by \cite[Proof of Theorem 5.1]{MW6}. We would like to understand the additional structure that the map $\Phi_\mathcal{A} (\Sigma)$ in \eqref{phimap2d} corresponds to, in comparison to just $\cat_{\text{\normalfont \bfseries !}}(\Sigma)$.

By \cite[Corollary 3.12]{BH}, the category $\int_{\mathbb{S}^1 \times [0,1]}\cat$ has enough compact projective objects. In particular, by \cite[Lemma 3.5]{Brandenburg2014ReflexivityAD}, $\int_{\mathbb{S}^1 \times [0,1]}\cat$ is dualizable in $\mathsf{Pr}$. Its dual, that we denote by $(\int_{\mathbb{S}^1\times [0,1]} \mathcal{A})^*$, is given by the internal hom
	\begin{align} 
	\left(\int_{\mathbb{S}^1\times [0,1]} \mathcal{A} \right)^*= \underline{\mathsf{Hom}}_{\mathsf{Pr}}\left(\int_{\mathbb{S}^1\times [0,1]}\cat,\mathsf{Vect}\right) \ .
	\end{align} 
Therefore, the map $\Phi_\mathcal{A} (\Sigma)$ is an object in the dual of $(\int_{\mathbb{S}^1\times [0,1]}\cat)^{\boxtimes n}$. In the following, we show that this amounts to an object in the Drinfeld center $\mathcal{Z}(\cat)$ of $\cat$.

The factorization homology $\int_{\mathbb{S}^1\times [0,1]}\cat$ is, by excision, the relative tensor product $\cat \boxtimes_{\cat^e}\cat$ where $\cat^e:=\cat^{\otimes \mathsf{op}}\boxtimes \cat$ and $\cat^{\otimes \mathsf{op}}$ is $\cat$ with the opposite monoidal product. Then we have
\begin{align}
	\left(\int_{\mathbb{S}^1\times [0,1]} \mathcal{A} \right)^* &\simeq \underline{\mathsf{Hom}}_{\mathsf{Pr}}\left(\int_{\mathbb{S}^1\times [0,1]}\cat,\mathsf{Vect}\right) \\
	&\simeq \underline{\mathsf{Hom}}_{\mathsf{Pr}}(\cat \boxtimes_{\cat^e}\cat,\mathsf{Vect}) \\
	&\simeq \underline{\mathsf{Hom}}_{\cat^e}(\cat, \underline{\mathsf{Hom}}_{\mathsf{Pr}}(\cat,\mathsf{Vect})) \\ &\simeq \underline{\mathsf{Hom}}_{\cat^e}(\cat, \cat^*) \ , \label{eqdualofsone}
\end{align}
here $\underline{\mathsf{Hom}}_{\cat^e}(-,-)$ denotes the left $\cat^e$-module maps. The self-duality of $\cat$ yields a map $\cat \simeq \cat^*$ that is given by $X \mapsto \kappa(X,-) \cong \kappa(-,X)$. It is easy to see that this is an $\cat^e$-module (or equivalently $\cat$-bimodule) equivalence. This implies
\begin{align}
	\left(\int_{\mathbb{S}^1\times [0,1]} \mathcal{A} \right)^* \simeq \underline{\mathsf{Hom}}_{\cat^e}(\cat, \cat)
\end{align}
and the right hand side is equivalent to the Drinfeld center $\mathcal{Z}(\cat)$ by \cite[Proposition 7.13.8]{EGNO}, see also \cite[Proposition 3.6]{BJSS}. In summary, $\Phi_{\cat}(\Sigma)$ from \eqref{phimap2d} amounts to an object in $\mathcal{Z}(\cat)^{\boxtimes n}$. In other words, the additional structure provided by $\Phi_{\cat}(\Sigma)$, in comparison to $\cat_{\text{\normalfont \bfseries !}}(\Sigma)$, can be expressed in terms of half-braidings.

\needspace{15\baselineskip}
\section{Open correlators}
Let us fix a category $\cat \in \mathsf{Pr}$ which satisfies the conditions spelled out in Proposition \ref{prop}. In this last section, we exhibit a class of open correlators for the open modular functor $\cat_{\text{\normalfont \bfseries !}}$. Recall that \emph{a system of open correlators} is a collection of vectors $\xi_{\Sigma}^F \in \cat_{\text{\normalfont \bfseries !}}(\Sigma;F,\dots,F)$ that are invariant with respect to the $\mathsf{Map}(\Sigma)$-action and that respect gluing. The notion of an open correlator is formalized in \cite[Section~8]{Woi} by extending the microcosm principle in the sense of \cite{BaDo} to cyclic and modular algebras. The microcosm principle provides a systematic way to construct operadic algebras \emph{inside} an algebra over the same operadic structure. Within this framework, an open correlator is defined as a modular $\mathsf{O}$-algebra inside $\cat_{\text{\normalfont \bfseries !}}$. 

The main result of this section is a characterization of modular $\mathsf{O}$-algebras in $\cat_{\text{\normalfont \bfseries !}}$ whose underlying object is compact projective. We will first characterize compact projective cyclic associative algebras inside $\cat$. Then we will use the fact that this can be uniquely extended to $\cat_{\text{\normalfont \bfseries !}}$. 

\subsection{Self-duality on compact projective objects}
The cyclic and modular microcosm principle is defined on a so-called \emph{self-dual object} $F \in \cat$, we refer to \cite{Woi} for the precise definition and motivation. The following is a characterization of self-duality on compact projective objects and  generalizes \cite[Lemma 7.1]{Woi}:

\begin{proposition}\label{ndsymp}
	A self-dual structure on a compact projective object $F \in \cat$ amounts to a pairing $\beta:F \otimes F \to I$ that is
	\begin{itemize}
		\item non-degenerate in the sense that the map $\psi: F \to F^\vee$ induced by $\Hom(F\otimes F,I) \cong \Hom(F,F^\vee)$ is an isomorphism,
		\item symmetric in the sense that $F \xrightarrow{\text{pivotality}} F^{\vee \vee} \xrightarrow{\psi^\vee} F^\vee$ coincides with $\psi$.
	\end{itemize}
\end{proposition}
\begin{proof}
	A self-dual structure on $F$ is a morphism $\gamma:F \boxtimes F \to \Delta$ that is subject to symmetry and non-degeneracy conditions \cite[Definition 4.5]{Woi}. In particular, it gives rise to a map $\psi:F \to F^\vee$ and a pairing $\beta: F \otimes F \to I$ via
	\begin{align}
		\mathsf{Hom}_{\cat \boxtimes \cat}(F \boxtimes F, \Delta) \cong \Hom(F,\Delta') \otimes \Hom(F,\Delta'') \cong \Hom(F,-)(\kappa(F^\vee,\Delta')\otimes \Delta'') \\ \cong \Hom(F,F^\vee) \cong \Hom(F\otimes F,I) \ .
	\end{align}
	Let us spell out the non-degeneracy and symmetry conditions.
	\begin{itemize}
		\item The non-degeneracy means that there exists a map $\delta:k \to \kappa(F,F)$ such that the compositions
		\begin{align}
			F \xrightarrow{\delta \otimes F} \kappa(F,F) \otimes F \xrightarrow{(\kappa \boxtimes \text{id}) (\text{id} \boxtimes \gamma)}  (\kappa \otimes \text{id}) (F \boxtimes \Delta) \cong F \label{nondegen} \ , \\
			F \xrightarrow{F \otimes \delta} F \otimes  \kappa(F,F) \xrightarrow{(\text{id} \boxtimes \kappa) (\gamma \boxtimes \text{id})}  ( \text{id} \otimes \kappa) (\Delta \boxtimes F) \cong F \ , \label{nondegen2}
		\end{align}
		are required to be the identity of $F$. 
		The map $\delta:k \to \kappa(F,F)$ defines a morphism $\phi:F^\vee \to F$ and a copairing $\delta:I \to F \otimes F$, that we again denote by $\delta$, via
		\begin{align}
			k \to \kappa(F,F) \cong \Hom(I,F \otimes F) \cong \Hom(F^\vee,F) \ .
		\end{align}
		Now, the condition in \eqref{nondegen} amounts to the map
		\begin{align}
			F \xrightarrow{\phi \otimes F} \Hom(F^\vee,F) \otimes F \xrightarrow{\Hom(F^\vee,\psi)\boxtimes F} \quad &\Hom(F^\vee,F^\vee) \otimes F \\ \to &\int^{X \in \mathsf{cp}(\cat)}\Hom(F^\vee,X^\vee) \otimes X \\ \cong &\int^{X \in \mathsf{cp}(\cat)}\Hom(X,F) \otimes X \cong F
		\end{align}
		being the identity morphism of $F$. This just means that the composition $\phi \circ \psi$ is $\text{id}_{F^\vee}$. Likewise, the condition \eqref{nondegen2} means that the composition
		\begin{align}
		 F \xrightarrow{F \otimes \phi} F \otimes \Hom(F^\vee,F)  \xrightarrow{\psi \otimes \text{id} }\quad & F^\vee \otimes \Hom(F^\vee,F)  \\ \to & \int^{X \in \mathsf{cp}(\cat)}X^\vee \otimes \Hom(X^\vee,F) \\  \cong &\int^{X \in \mathsf{cp}(\cat)}X \otimes \Hom(X,F) \cong F
		\end{align}
		
		is the identity of $F$. This just means that the composition $\psi \circ \phi$ is $\text{id}_{F}$. Consequently, the non-degeneracy of $\gamma$ is equivalent to $\psi$ being an isomorphism.
	
		\item The symmetry condition means that $\gamma$ is fixed by the $\mathbb{Z}_2$-action on $\mathsf{Hom}_{\cat \boxtimes \cat}(F \boxtimes F,\Delta)$. Equivalently, this means that $\psi$ is fixed by the map $\Hom(F,F^\vee) \cong \Hom(F^{\vee\vee},F^\vee) \cong \Hom(F,F^\vee)$, where the first map is given by taking duals and the second map is induced by the pivotal structure. This means that the composition $F \cong F^{\vee\vee} \xrightarrow{\psi^\vee} F^\vee$ is required to be $\psi$ and we find the condition in the statement.
	\end{itemize}
\end{proof}

\subsection{Frobenius structures}
The characterization of self-duality on compact projective objects leads to the following description of compact projective cyclic associative algebras in terms of symmetric Frobenius algebras.

\begin{proposition} \label{propcycas}
	A compact projective cyclic associative algebra $F$ in $\cat$ is a symmetric Frobenius algebra in $\cat$, i.e.\ $F$ is 
	\begin{itemize}
		\item a unital associative algebra with the unit $\eta:I \to F$ and the multiplication $\mu:F \otimes F \to F$,
		\item equipped with a non-degenerate symmetric pairing $\beta: F\otimes F \to I$ in the sense of Proposition \ref{ndsymp} that is invariant, i.e.\ $\beta(\eta,\mu)=\mu$.
	\end{itemize}
\end{proposition}
\begin{proof}
	The non-degenerate symmetric pairing $\beta$ endows $F$ with a self-dual structure by Proposition \ref{ndsymp}. By \cite[Remark 5.4]{Woi}, the cyclic associative algebra structure on $F$ is determined by the generating operations and relations of the cyclic associative operad. Namely, each generating operation of the associative operad induces a structure map on $F$, while the relations tell us the conditions that need to be satisfied by the structure maps. By following the presentatation of the cyclic associative operad given in \cite[Definition 4.1]{MW1}, we see that the two generating operations of the associative operad, namely the product and the unit, define two morphisms $\mu:F \otimes F \to F$ and $\eta:I \to F$. The associativity and unit isomorphisms of $\otimes$ entails that $\mu$ and $\eta$ endow $F$ with a unital associative algebra in $\cat$. Lastly, the relation (Z) in \cite[Definition 4.1]{MW1}, i.e.\ the isomorphism $\kappa(I,-\otimes-) \cong \kappa(-,-)$, implies that $\beta$ is invariant.
\end{proof}

With this result, we obtain the following examples of open correlators:

\begin{corollary} \label{corr}
	A compact projective symmetric Frobenius algebra $F \in \cat$ defines a consistent system of open correlators $\xi_\Sigma^F \in \cat_{\text{\normalfont \bfseries !}}(\Sigma;F,\dots,F)$. Conversely, any such system with underlying compact projective object $F$ defines a symmetric Frobenius algebra structure on $F$.
\end{corollary}
\begin{proof}
	This follows from a microcosmic version of the equivalence between cyclic associative algebras and open modular functors. More precisely, the collection of cyclic associative algebras in $\cat$ forms a groupoid \cite[Remark 5.3]{Woi}. Likewise, the collection of open correlators for $\cat_{\text{\normalfont \bfseries !}}$ forms a groupoid and these two groupoids are equivalent  \cite[Theorem 6.1]{Woi}. Now the result follows from Proposition \ref{propcycas}.
\end{proof}

\begin{exxample}\label{excor}
	In Example~\ref{examplevoa}, any projective object $P\in \mathscr{F}$ in the module category for the $\beta\gamma$ ghost produces a compact projective symmetric Frobenius algebra $P\otimes P^\vee$ in the ind completion $\mathbf{F}$ of $\mathscr{F}$. This gives rise to a system of 
	open correlators via Corollary~\ref{corr} with a holographic description. The structure maps of the symmetric Frobenius algebra $P\otimes P^\vee$ is given by suitable compositions of evaluation and coevaluation maps.
	\end{exxample}

\begin{exxample}[Towards a non-finite version of the Cardy case]
	Corollary~\ref{corr} can also be applied to some non-finite module categories
	of vertex operator algebras that are in particular semisimple, see again the list in the introduction of \cite{Creutzig:2024qvb} for examples. In that situation, the monoidal unit endowed with the standard symmetric Frobenius algebra structure provides at least an open correlator. Choosing the monoidal unit as a boundary condition is referred to as the \emph{Cardy case}, see e.g.~\cite{FGSS} for more background.
	\end{exxample}

\noindent \textsc{Université Bourgogne Europe, CNRS, IMB UMR 5584, F-21000 Dijon, France}

\end{document}

%% file: styles.tikzstyles
\tikzstyle{thickblue}=[-,thick, draw=blue]
\tikzstyle{black}=[-, draw=black]
\tikzstyle{blue}=[-, draw=blue]
\tikzstyle{dashed}=[-,dotted,draw=black]
\tikzstyle{dashedred}=[-,dotted,dashed,draw=red]
\tikzstyle{smallblue}=[fill=blue, draw=blue, shape=circle, minimum size=4pt, inner sep=0pt]
\tikzstyle{bluecircle}=[fill=white, draw=blue, shape=circle]
\tikzstyle{endarrow}=[->]

%% file: non_finite_open_modular_functors.bbl
\medskip
\small
\begin{thebibliography}{MSWY23}
	
	\bibitem[AF15]{AF}
	D.~Ayala and J.~Francis.
	\newblock Factorization homology of topological manifolds.
	\newblock {\em J. Top.}, 8(4):1045--1084, 2015.
	
	\bibitem[AW22]{AW}
	R.~Allen and S.~Wood.
	\newblock Bosonic ghostbusting: The bosonic ghost vertex algebra admits a
	logarithmic module category with rigid fusion.
	\newblock {\em Comm. Math. Phys.}, 390:959--1015, 2022.
	
	\bibitem[BCJF14]{Brandenburg2014ReflexivityAD}
	M.~Brandenburg, A.~Chirvasitu, and T.~Johnson-Freyd.
	\newblock Reflexivity and dualizability in categorified linear algebra.
	\newblock {\em Theory Appl. Cat.}, 2014.
	
	\bibitem[BD98]{BaDo}
	J.~Baez and J.~Dolan.
	\newblock Higher-dimensional algebra iii. $n$-categories and the algebra of
	opetopes.
	\newblock {\em Adv. Math.}, 135(2):145--206, 1998.
	
	\bibitem[BD04]{BeDr}
	A.~Beilinson and V.~Drinfeld.
	\newblock {\em Chiral Algebras}, volume~51 of {\em Colloquium Publications}.
	\newblock Amer. Math. Soc., 2004.
	
	\bibitem[BH24]{BH}
	J.~Brown and B.~Ha\"ioun.
	\newblock Skein categories in non-semisimple settings.
	\newblock To appear in \emph{Selecta Math.} arXiv:2406.08956 [math.QA], 2024.
	
	\bibitem[BJ17]{BJ}
	A.~Brochier and D.~Jordan.
	\newblock {Fourier transform for quantum $D$-modules via the punctured torus
		mapping class group}.
	\newblock {\em Quantum Top.}, 8:361--379, 2017.
	
	\bibitem[BJS21]{BJS}
	A.~Brochier, D.~Jordan, and N.~Snyder.
	\newblock On dualizability of braided tensor categories.
	\newblock {\em Compositio Math.}, 157(3):435–483, 2021.
	
	\bibitem[BJSS21]{BJSS}
	A.~Brochier, D.~Jordan, P.~Safronov, and N.~Snyder.
	\newblock Invertible braided tensor categories.
	\newblock {\em Alg. Geom. Top.}, 21(4):2107--2140, 2021.
	
	\bibitem[BK01]{BK}
	B.~Bakalov and A.~Kirillov.
	\newblock {\em Lectures on tensor categories and modular functors}, volume~21
	of {\em University Lecture Series}.
	\newblock Amer. Math. Soc., 2001.
	
	\bibitem[BSZ25]{BSZ}
	S.~Barkan, J.~Steinebrunner, and A.~Y. Zhang.
	\newblock Open 2{D} {TFT}s admit initial open-closed extensions.
	\newblock arXiv:2509.02553 [math.AT], 2025.
	
	\bibitem[Bud08]{Bud}
	R.~Budney.
	\newblock The operad of framed discs is cyclic.
	\newblock {\em J. Pure Appl. Alg.}, 212(1):193--196, 2008.
	
	\bibitem[BW22]{BW}
	A.~Brochier and L.~Woike.
	\newblock A classification of modular functors via factorization homology.
	\newblock Accepted for publication in \emph{Geom. Top.} arXiv:2212.11259
	[math.QA], 2022.
	
	\bibitem[BZBJ18]{BZBJ}
	D.~Ben-Zvi, A.~Brochier, and D.~Jordan.
	\newblock {Integrating quantum groups over surfaces}.
	\newblock {\em J. Top.}, 11(4):874--917, 2018.
	
	\bibitem[CGPM23]{CGP}
	F.~Costantino, N.~Geer, and B.~Patureau-Mirand.
	\newblock Admissible skein modules.
	\newblock arXiv:2302.04493 [math.GT], 2023.
	
	\bibitem[CMSY24]{Creutzig:2024qvb}
	T.~Creutzig, R.~McRae, K.~Shimizu, and H.~Yadav.
	\newblock Commutative algebras in {G}rothendieck-{V}erdier categories,
	rigidity, and vertex operator algebras.
	\newblock arXiv:2409.14618 [math.QA], 2024.
	
	\bibitem[Coo23]{Cooke}
	J.~Cooke.
	\newblock {Excision of skein categories and factorisation homology}.
	\newblock {\em Adv. Math}, 414:108848, 2023.
	
	\bibitem[Cos04]{Cos}
	K.~Costello.
	\newblock The {A}-infinity operad and the moduli space of curves.
	\newblock arXiv:math/0402015 [math.AG], 2004.
	
	\bibitem[Cos07]{Costello:2004ei}
	K.~Costello.
	\newblock {Topological conformal field theories and {C}alabi-{Y}au categories}.
	\newblock {\em Adv. Math.}, 210:165--214, 2007.
	
	\bibitem[EGNO15]{EGNO}
	P.~Etingof, S.~Gelaki, D.~Nikshych, and V.~Ostrik.
	\newblock {\em Tensor categories}, volume 205 of {\em Math. Surveys and
		Monogr.}
	\newblock Amer. Math. Soc., 2015.
	
	\bibitem[FGSS18]{FGSS}
	J.~Fuchs, T.~Gannon, G.~Schaumann, and C.~Schweigert.
	\newblock {The logarithmic {C}ardy case: Boundary states and annuli}.
	\newblock {\em Nucl. Phys. B}, 930:287--327, 2018.
	
	\bibitem[FSWY25]{algcften}
	J.~Fuchs, C.~Schweigert, S.~Wood, and Y.~Yang.
	\newblock Algebraic structures in two-dimensional conformal field theory.
	\newblock In R.~Szabo and M.~Bojowald, editors, {\em Encyclopedia of
		Mathematical Physics (Second Edition)}, pages 604--617. Academic Press, 2025.
	
	\bibitem[Gia11]{Gia}
	J.~Giansiracusa.
	\newblock The framed little 2-discs operad and diffeomorphisms of handlebodies.
	\newblock {\em J. Top.}, 4(4):919--941, 2011.
	
	\bibitem[GK95]{GK1}
	E.~Getzler and M.~Kapranov.
	\newblock Cyclic operads and cyclic homology.
	\newblock In {\em Geometry, Topology, and Physics}, pages 167--201.
	International Press, 1995.
	
	\bibitem[GK98]{GK2}
	E.~Getzler and M.~Kapranov.
	\newblock Modular operads.
	\newblock {\em Compositio Math.}, 110:65--125, 1998.
	
	\bibitem[Laz01]{LAZAROIU2001497}
	C.I. Lazaroiu.
	\newblock On the structure of open–closed topological field theory in two
	dimensions.
	\newblock {\em Nucl. Phys. B}, 603(3):497--530, 2001.
	
	\bibitem[Lur]{Lur}
	J.~Lurie.
	\newblock Higher algebra.
	\newblock https://people.math.harvard.edu/~lurie/papers/HA.pdf.
	
	\bibitem[MS89]{MS}
	G.~Moore and N.~Seiberg.
	\newblock Classical and quantum conformal field theory.
	\newblock {\em In IX International Conference on Mathematical Physics (IAMP)},
	1989.
	
	\bibitem[MS06]{Moore:2006dw}
	G.~Moore and G.~Segal.
	\newblock {D-branes and {K}-theory in 2{D} topological field theory}.
	\newblock arXiv:hep-th/0609042, 2006.
	
	\bibitem[MSWY23]{MSWY}
	L.~Müller, C.~Schweigert, L.~Woike, and Y.~Yang.
	\newblock The {L}yubashenko modular functor for {D}rinfeld centers via
	non-semisimple string-nets.
	\newblock Accepted for publication in \emph{Adv. Math.} arXiv:2312.14010
	[math.QA], 2023.
	
	\bibitem[MW23]{MW1}
	L.~Müller and L.~Woike.
	\newblock Cyclic framed little disks algebras, {G}rothendieck–{V}erdier
	duality and handlebody group representations.
	\newblock {\em Quart. J. Math.}, 74(1):163--245, 2023.
	
	\bibitem[MW24a]{MW3}
	L.~Müller and L.~Woike.
	\newblock Admissible skein modules and ansular functors: A comparison.
	\newblock Accepted for publication in \emph{Ann. Fac. Sci. Tou.}
	arXiv:2409.17047 [math.QA], 2024.
	
	\bibitem[MW24b]{MW2}
	L.~Müller and L.~Woike.
	\newblock Classification of consistent systems of handlebody group
	representations.
	\newblock {\em Int. Math. Res. Not.}, 2024(6):4767--4803, 2024.
	
	\bibitem[MW25]{MW6}
	L.~M{\"u}ller and L.~Woike.
	\newblock {Categorified Open Topological Field Theories}.
	\newblock {\em Proc. Amer. Math. Soc.}, 153:2381--2396, 2025.
	
	\bibitem[RST24]{RST}
	I.~Runkel, C.~Schweigert, and Y.~H. Tham.
	\newblock Excision for spaces of admissible skeins.
	\newblock arXiv:2407.09302 [math.QA], 2024.
	
	\bibitem[Seg88]{Seg}
	G.~Segal.
	\newblock Two-dimensional conformal field theories and modular functors.
	\newblock {\em Comm. Math. Phys.}, 1988.
	
	\bibitem[SP09]{SP}
	C.~J. Schommer-Pries.
	\newblock {\em The classification of two-dimensional extended topological field
		theories}.
	\newblock PhD thesis, Berkeley, 2009.
	
	\bibitem[SW03]{SalWahl}
	P.~Salvatore and N.~Wahl.
	\newblock Framed discs operads and {B}atalin–{V}ilkovisky algebras.
	\newblock {\em Quart. J. Math.}, 54(2):213--231, 2003.
	
	\bibitem[Til98]{Til}
	U.~Tillmann.
	\newblock S-structures for k-linear categories and the definition of a modular
	functor.
	\newblock {\em J. London Math. Soc.}, 58(1):208--228, 08 1998.
	
	\bibitem[Tur94]{Tur}
	V.~G. Turaev.
	\newblock {\em Quantum Invariants of Knots and 3-Manifolds}, volume~18 of {\em
		Studies in Math.}
	\newblock De Gruyter, 1994.
	
	\bibitem[Wah01]{Wahl}
	N.~Wahl.
	\newblock {\em Ribbon braids and related operads}.
	\newblock PhD thesis, Oxford, 2001.
	
	\bibitem[Wal24]{Walton}
	C.~Walton.
	\newblock {\em Symmetries of Algebras}, volume~1.
	\newblock 619 Wreath Publishing, 2024.
	
	\bibitem[Woi25a]{Woi3}
	L.~Woike.
	\newblock The construction of correlators in finite rigid logarithmic conformal
	field theory.
	\newblock arXiv:2507.22841 [math.QA], 2025.
	
	\bibitem[Woi25b]{Woi}
	L.~Woike.
	\newblock The cyclic and modular microcosm principle in quantum topology.
	\newblock {\em Canad. J. Math.}, pages 1--41, 2025.
	
	\bibitem[Woi25c]{Woi2}
	L.~Woike.
	\newblock Reflection equivariance and the {H}eisenberg picture for spaces of
	conformal blocks.
	\newblock arXiv:2507.22820 [math.QA], 2025.
	
\end{thebibliography}
